\pdfoutput=1
\RequirePackage{ifpdf}
\ifpdf 
\documentclass[pdftex]{sigma}
\else
\documentclass{sigma}
\fi

\usepackage{mathtools}
\usepackage{enumerate}

\numberwithin{equation}{section}

\newtheorem{Theorem}{Theorem}[section]
\newtheorem{Corollary}[Theorem]{Corollary}
\newtheorem{Lemma}[Theorem]{Lemma}
\newtheorem{Proposition}[Theorem]{Proposition}
 { \theoremstyle{definition}
\newtheorem{Definition}[Theorem]{Definition}
\newtheorem{Example}[Theorem]{Example}
\newtheorem{Remark}[Theorem]{Remark} }

\begin{document}

\allowdisplaybreaks

\newcommand{\arXivNumber}{1610.09898}

\renewcommand{\thefootnote}{}

\renewcommand{\PaperNumber}{022}

\FirstPageHeading

\ShortArticleName{$G$-Invariant Deformations of Almost-Coupling Poisson Structures}

\ArticleName{$\boldsymbol{G}$-Invariant Deformations of Almost-Coupling\\ Poisson Structures\footnote{This paper is a~contribution to the Special Issue ``Gone Fishing''. The full collection is available at \href{http://www.emis.de/journals/SIGMA/gone-fishing2016.html}{http://www.emis.de/journals/SIGMA/gone-fishing2016.html}}}

\Author{Jos\'e Antonio VALLEJO~$^\dag$ and Yury VOROBIEV~$^\ddag$}

\AuthorNameForHeading{J.A.~Vallejo and Yu.~Vorobiev}

\Address{$^\dag$~Facultad de Ciencias, Universidad Aut\'onoma de San Luis Potos\'i, M\'exico}
\EmailD{\href{mailto:jvallejo@fc.uaslp.mx}{jvallejo@fc.uaslp.mx}}
\URLaddressD{\url{http://galia.fc.uaslp.mx/~jvallejo/}}

\Address{$^\ddag$~Departamento de Matem\'aticas, Universidad de Sonora, M\'exico}
\EmailD{\href{mailto:yurimv@guaymas.uson.mx}{yurimv@guaymas.uson.mx}}

\ArticleDates{Received October 31, 2016, in f\/inal form March 28, 2017; Published online April 02, 2017}

\Abstract{On a foliated manifold equipped with an action of a compact Lie group $G$, we study a class of almost-coupling Poisson and Dirac structures, in the context of deformation theory and the method of averaging.}


\Classification{53D17; 70G45; 58H15}

\renewcommand{\thefootnote}{\arabic{footnote}}
\setcounter{footnote}{0}

\section{Introduction}

In this paper, we develop further the results of \cite{JVYV-14,Vor-08}, on the construction of invariant Poisson and Dirac structures via the averaging method on foliated Poisson manifolds with symmetry in the context of deformation theory.

For a Poisson bivector f\/ield $\Pi$ on a foliated manifold $(M,\mathcal{F})$, the $\mathcal{F}$-\textit{almost-coupling proper\-ty}~\cite{Va-04} means the existence of a~normal bundle structure $\mathbb{H}$ of $\mathcal{F}$, such that in the $\mathbb{H}$-dependent bigraded decomposition of $\Pi$, the mixed term of bidegree $(1,1)$ vanishes. In particular, such class of Poisson structure contains the $\mathcal{F}$-coupling Poisson structures which play important roles in some problems of semi-local Poisson geometry \cite{CrFe,CrMa-12,Va-04, Vo-01}. Moreover, the (almost) coupling constructions can be naturally extended to the Dirac category \cite{BrF, Va-06,Wa}.

Now, starting with a foliated Poisson manifold $(M,\mathcal{F},P)$, equipped with a~leaf preserving action of a compact connected Lie group $G$, our point is to study, in the context of the averaging method, some deformations $\{\Pi_{\varepsilon }\}_{\varepsilon\in [0,1]}$ of the leaf-tangent\footnote{Following \cite{Va-06}, a~leaf-tangent (Poisson) bivector f\/ield on a foliated manifold is a section of $\bigwedge^2 T\mathcal{F}$.} Poisson bivector f\/ield $P$ in the class of the $\mathcal{F}$-\textit{almost-coupling Poisson structures} $\Pi_{\varepsilon}$. Our approach is based on the averaging technique~\cite{JVYV-14} related to a class of exact gauge transformations for Poisson and Dirac structures. The idea of the construction of $G$-invariant Poisson structures is to follow the path:
\begin{gather*}
\text{Poisson}\xrightarrow[]{\text{averaging}}G\text{-invariant Dirac}\xrightarrow[]{\text{non-degeneracy}}G\text{-invariant Poisson}.
\end{gather*}

For the case of a $G$-action which is locally Hamiltonian relative to $P$, we give some results about the realization of the above scheme for a class of $\mathcal{F}$-almost-coupling Poisson deformations of $P$. Within the framework of perturbation theory, these results can be applied to the study of invariant normal forms for Hamiltonian systems of adiabatic type \cite{ArKN-88}, associated to deformations of Poisson structures \cite{MJYu-13,MYU-16}.

In a generalized setting, by a Hamiltonian system of adiabatic type on a Poisson foliation $(M,\mathcal{F},P)$ with $G$-symmetry, we mean a Hamiltonian system relative to a deformed Poisson structure $\Pi_{\varepsilon}=P+\varepsilon \Lambda$ and a function $F\in C^{\infty}(M)$. The deformation of $P$ is given by a bivector f\/ield $\Lambda\in\Gamma(\wedge^{2}\mathbb{H})$, where $\mathbb{H}\subset TM$ is a normal bundle of the foliation $\mathcal{F}$. In the (adiabatic) limit $\varepsilon\to 0$, the unperturbed Hamiltonian system $(M,P,F)$ is $G$-invariant and describes the fast dynamics along the leaves of $\mathcal{F}$. The key point is to move the original perturbed system to a~$G$-invariant one by using a near-identity transformation. The construction of the corresponding $G$-invariant model is related to the averaging procedure for the deformed Poisson structure~$\Pi_{\varepsilon}$. In a particular setting this problem was studied in~\cite{MYU-16}, here we present a general mechanism.

\section{Preliminaries}

Let us recall some basic facts that will be used later on, related to the averaging procedure with respect to the action of compact Lie groups, for Poisson and Dirac structures~\cite{JVYV-14}.

\subsection{Gauge transformations of Dirac manifolds}
Let $(M,D)$ be a Dirac manifold, that is, a smooth regular distribution $D\subset TM\oplus T^{\ast}M$ which is maximally isotropic with respect to the natural symmetric pairing on $TM\oplus T^{\ast}M$, and involutive with respect to the Courant bracket \cite{Cou-90,CuWe-86}. The Dirac manifold $(M,D)$ carries a (singular) presymplectic foliation $(\mathcal{S},\omega)$: Its leaves are the maximal integral manifolds of the integrable (singular) distribution $C:=p_{T}(D)\subset TM$, and the leafwise presymplectic structure $\omega$ is def\/ined by $\omega_{q}(X,Y)=-\alpha(Y)$, for $(X,Y)\in C_{q}:=C\cap T_qM$ and $(X,\alpha)\in D_{q}$, $q\in M$. Here \smash{$p_{T}\colon TM\oplus T^{\ast}M$} $\rightarrow TM$ is the natural projection.

We can modify the leafwise presymplectic structure $\omega$ by the pull back of a~closed $2$-form on the base $B\in\Omega^{2}(M)$: For each presymplectic leaf $(S,\omega_{S})$, we def\/ine the new presymplectic structure as $\omega_{S}+\iota_{S}^{\ast}B$, where $\iota_{S}\colon S\hookrightarrow M$ is the inclusion map. Then, the foliation $\mathcal{S}$ equipped with the deformed leafwise presymplectic structure gives rise to the new Dirac structure
\begin{gather*}
\tau_{B}(D)=\{(X,\alpha-\mathbf{i}_{X}B)\,|\,(X,\alpha)\in D\}.
\end{gather*}
This transformation $\tau_{B}$ preserves the presymplectic foliation of $D$, and is called the \emph{gauge transformation} associated to the closed $2$-form $B$ \cite{BurRa-03, SeWe-01}.

In particular, the foliation $(S,\omega)$ is symplectic if and only if $D$ is the graph of a Poisson bivector f\/ield $\Pi$ on $M$,
\begin{gather}
D=\operatorname*{Graph}\Pi=\big\{\big(\Pi^{\natural}\alpha,\alpha\big)\,|\,\alpha\in\Omega^{1}(M)\big\}, \label{H1}
\end{gather}
where $\Pi^{\natural}\colon T^{\ast}M\rightarrow TM$ is the induced vector bundle endomorphism given by $\alpha\mapsto\mathbf{i}_{\alpha}\Pi$. Condition~\eqref{H1} can be expressed as follows
\begin{gather*}
D\cap  ( TM\oplus\{0\} ) =\{(0,0)\}.
\end{gather*}
Notice that, in general, for a given closed $2$-form $B$, the gauge transformation $\tau_{B}$ takes $\operatorname*{Graph}\Pi$ to another Dirac structure
\begin{gather*}
\tau_{B}(\operatorname*{Graph}\Pi)=\big\{\big(\Pi^{\natural}\alpha,\alpha-\mathbf{i}_{\Pi^{\natural}\alpha}B\big)\,|\,\alpha\in\Omega^{1}(M)\big\},
\end{gather*}
which may not necessarily come from a Poisson bivector f\/ield. We have the simple, but useful, criterion (see \cite[pp.~5--6]{SeWe-01}).

\begin{Lemma}\label{lem21}
If the endomorphism
\begin{gather*}
\big( \operatorname{Id}-B^{\flat}\circ\Pi^{\natural}\big) \colon \ T^{\ast }M\rightarrow T^{\ast}M
\end{gather*}
is invertible, then the Dirac structure $\tau_{B}(\mathrm{Graph}\Pi)$ is the graph of the Poisson tensor $\tau_{B}(\Pi)$ whose induced endomorphism is
defined by
\begin{gather*}
\tau_{B}(\Pi)^{\sharp}=\Pi^{\sharp}\circ\big(\operatorname{Id}-B^{\flat}\circ \Pi^{\sharp}\big)^{-1}.
\end{gather*}
\end{Lemma}

 Here, we denote by $B^{\flat}\colon TM\to T^*M$ the vector bundle endomorphism given by $X\mapsto \mathbf{i}_XB$.

\subsection[$G$-averaging procedure]{$\boldsymbol{G}$-averaging procedure}
Let $G$ be a connected, compact Lie group, and $\mathfrak{g}$ its Lie algebra. Suppose we are given a smooth (left) action $\Phi\colon G\times M\rightarrow M$ on a manifold $M$. For every $a\in\mathfrak{g}$, the corresponding inf\/initesimal generator is denoted by $a_{M}\in\mathfrak{X}(M)$,
\begin{gather*}
a_{M}(q):=\left.\frac{\mathrm{d}}{\mathrm{d}t}\right|_{t=0}\Phi_{\exp(ta)}(q),\qquad q\in M.
\end{gather*}
For every tensor f\/ield $T$ on $M$, we denote by $ \langle T \rangle^{G}$ its $G$-average, which is def\/ined as
\begin{gather*}
 \langle T \rangle^{G}:=\int_G \Phi^{\ast}_{g}T\,\mathrm{d}g,
\end{gather*}
where $\mathrm{d}g$ is the normalized Haar measure on $G$. This averaging procedure can be also applied to Dirac structures \cite{JVYV-14}: Let $D\subset TM\oplus T^{\ast}M$ be a Dirac structure on $M$, and let $(\mathcal{S},\omega)$ be its associated presymplectic foliation, carrying the leafwise presymplectic form~$\omega$.

\begin{Definition}\label{def2.1}
The $G$-action on $M$ is \emph{compatible} with the Dirac structure $D$ if each leaf $S$, of~$\mathcal{S}$, is invariant under the action of $G$, and there exists a~$\mathbb{R}$-linear mapping $\rho\in\mathrm{Hom}(\mathfrak{g},\Omega^{1}(M))$ such that
\begin{gather}
\mathbf{i}_{a_{M}}\omega_S=-\iota^\ast_S\rho_{a}, \label{LW1}
\end{gather}
for every $a\in\mathfrak{g}$ (here, $\omega_S$ is a presymplectic structure on $S$, and $\iota_S\colon S\hookrightarrow M$ is the canonical injection).
\end{Definition}

Equivalently, condition \eqref{LW1} can be rewritten as follows
\begin{gather*}
(a_{M},\rho_{a})\in\Gamma(D)\mbox{ for all }a\in\mathfrak{g}. 
\end{gather*}
Then, as a consequence of the fact that $G$ is compact connected (hence the exponential map is surjective) we have the following fact~\cite{JVYV-14}.

\begin{Lemma}\label{lema22}
If the $G$-action is compatible with $D$, then there exists a Dirac structure~$\overline{D}$ on~$M$ with the following properties:
\begin{enumerate}[$(a)$]\itemsep=0pt
\item The leafwise presymplectic form of $\overline{D}$ is defined as
\begin{gather*}
\langle \omega \rangle^{G}_S:=\omega_S -\iota^*_S \mathrm{d}\Theta,
\end{gather*}
where $\Theta\in\Omega^{1}(M)$ is the $1$-form determined in terms of the $G$-action and $\rho$ by
\begin{gather}
\Theta:=\int_G \left(
\int^1_0 \Phi_{\exp(\tau a)}^{\ast}\rho_{a}\,\mathrm{d}\tau\right) \mathrm{d}g,\qquad g=\exp a.\label{MF}
\end{gather}
\item $\overline{D}$ is $G$-invariant:
\begin{gather*}
(X,\alpha)\in\Gamma(D)\mbox{ implies that }(\Phi_{g}^{\ast}X,\Phi_{g}^{\ast} \alpha)\in\Gamma(D)\mbox{ for all }g\in G.
\end{gather*}
\item The Dirac structure $\overline{D}$ is is related to $D$ by an exact gauge transformation:
\begin{gather*}
\overline{D}=\{(X,\alpha+\mathbf{i}_{X}\mathrm{d}\Theta)\,|\, (X,\alpha)\in D\}.
\end{gather*}
\end{enumerate}
\end{Lemma}

The Dirac structure $\overline{D}$ will be called the $G$-\emph{average} of $D$ relative to the compatible $G$-action. As we have mentioned above, for the case in which the Dirac structure is the graph of a Poisson tensor $\Pi$ on~$M$, $D=\operatorname*{Graph}\Pi$, its $G$-average $\overline{D}=\tau_{B}(D)$, where $B=-\mathrm{d}\Theta$, does not necessarily come from a Poisson structure. In other words: In general, the averaging of Poisson structures via compatible $G$-actions only leads to $G$-invariant Dirac structures; but, by Lemma~\ref{lem21}, in the particular case of an invertible endomorphism
\begin{gather*}
\operatorname{Id}+(\mathrm{d}\Theta)^{\flat}\circ\Pi^{\natural},
\end{gather*}
we can say something more, namely, that the $G$-average $\overline{D}$ is the graph of the $G$-invariant Poisson tensor $\overline{\Pi}=\tau_{B}(\Pi)$.

\subsection[$G$-invariant connections]{$\boldsymbol{G}$-invariant connections}
Recall that a \emph{$($generalized Ehresmann$)$ connection} on a manifold~$M$~\cite{KMS-93}, is a vector-valued $1$-form $\gamma\in\Omega^{1}(M;TM)$ satisfying the following conditions: $\gamma^{2}=\gamma$, and the rank of the distribution $\operatorname*{Im}\gamma\subset TM$ is constant on~$M$ (this last condition is a consequence of the f\/irst for $M$ connected). Assume that the action $\Phi\colon G\times M\rightarrow M$, of the compact, connected Lie group $G$ on the connected manifold $M$, preserves the image of~$\gamma$, $\mathrm{d}_{q}\Phi_{g}$ $\circ\gamma_{q}=\gamma_{\Phi_{g}(q)}$, for all $q\in M$, $g\in G$. Then, the $G$-average of $\gamma$ is the vector-valued $1$-form
$\langle \gamma \rangle^{G}\in$ $\Omega^{1}(M,TM)$ def\/ined by the formula
\begin{gather*}
\langle \gamma \rangle^{G}(X):= \int_{G}\Phi_{g}^{\ast}(\gamma((\Phi_{g})_{\ast}X)\,\mathrm{d}g,
\end{gather*}
where $X\in\mathfrak{X}(M)$ is any vector f\/ield, and this is a $G$-invariant connection on~$M$.

\section[$\mathcal{F}$-almost-coupling structures]{$\boldsymbol{\mathcal{F}}$-almost-coupling structures}

In this section, we present some basic def\/initions and facts concerning the coupling procedure on foliated Poisson and Dirac manifolds (for more details, see~\cite{Va-04,Va-06, Vo-01}). The term `coupling' comes from Sternberg's coupling method on symplectic f\/iber bundles (see, for
example,~\cite{GLS-96}).

Suppose we are given a manifold with a (regular) foliation $(M,\mathcal{F})$. We will denote by $\mathbb{V}:=T\mathcal{F}$ the tangent bundle of $\mathcal{F}$, and by $\mathbb{V}^{0}=\operatorname*{Ann}(T\mathcal{F})\subset T^{\ast}M$ its annihilator. By a normal bundle of the foliation $\mathcal{F}$ we mean a~sub-bundle $\mathbb{H}\subset TM$ which is complementary to the tangent bundle~$\mathbb{V}$,
\begin{gather}
TM=\mathbb{H}\oplus\mathbb{V}.\label{AC1}
\end{gather}
To a normal bundle $\mathbb{H}$, there is associated a vector-valued $1$-form $\gamma\in\Omega^{1}(M;TM)$, def\/ined as the canonical projection along $\mathbb{H}$, $\gamma := \operatorname{pr}_{\mathbb{H}}\colon TM\rightarrow\mathbb{V}$. This form satisf\/ies $\gamma^{2}=\gamma$ and $\operatorname{Im}\gamma=\mathbb{V}$, and hence def\/ines an Ehresmann connection on the foliated manifold $(M,\mathcal{F)}$ (see~\cite{KMS-93}). The curvature of the connection $\gamma$ is the vector-valued $2$-form $R^{\gamma}\in\Omega^{2}(M;\mathbb{V})$ given by $R^{\gamma}=\frac{1}{2}[\gamma,\gamma]_{\operatorname*{FN}}$, where $[\cdot ,\cdot ]_{\operatorname*{FN}}$ denotes the Fr\"{o}licher--Nijenhuis bracket for vector valued forms on~$M$~\cite{KMS-93}. The curvature controls the integrability of the normal sub-bundle, in the sense that $\mathbb{H}$ is integrable if and only if~$\gamma$ is f\/lat, $R^{\gamma}=0$.

According to \eqref{AC1}, we have the dual splitting
\begin{gather}
T^{\ast}M=\mathbb{V}^{0}\oplus\mathbb{H}^{0}. \label{Spl2}
\end{gather}
Then, \eqref{AC1} and \eqref{Spl2} induce an $\mathbb{H}$-dependent bigrading of multivector f\/ields and dif\/ferential forms on $M$. For any $A\in
\Gamma(\wedge^{k}TM)$ and $\alpha\in\Omega^{k}(M)$, we have $A=\sum\limits _{s+l=k}A_{s,l}$ and $\alpha=\sum\limits_{s+l=k}\alpha_{s,l}$, where the elements~$A_{s,l}$ and~$\alpha_{s,l}$, belonging to the subspaces $\Gamma(\wedge^{s}\mathbb{H})\otimes\Gamma(\wedge^{l}\mathbb{V})$ and $\Gamma(\wedge ^{s}\mathbb{V}^{0})\otimes\Gamma(\wedge^{l}\mathbb{H}^{0})$, respectively, are said to be multivector f\/ields and forms of bidegree $(s,l)$. Moreover, the exterior dif\/ferential $\mathrm{d}$ on $M$ inherits the $\mathbb{H}$-bigrading decomposition (see \cite{Va-04}) in the form
$\mathrm{d}=\mathrm{d}_{1,0}+\mathrm{d}_{2,-1}+\mathrm{d}_{0,1}$.

We will need the following def\/inition \cite{Va-04, Vo-01}.

\begin{Definition}
A \emph{Poisson structure} $\Pi$ on the foliated manifold $(M,\mathcal{F})$ is said to be $\mathcal{F}$-\emph{almost-coupling}, via a~normal bundle $\mathbb{H}$ of the foliation $\mathcal{F}$, if the image of~$\mathbb{V}^{0}$ under the vector bundle morphism $\Pi^{\sharp}\colon T^{\ast}M\rightarrow TM$, is contained in~$\mathbb{H}$:
\begin{gather}\label{AC3}
\Pi^{\sharp}\big(\mathbb{V}^{0}\big)\subseteq\mathbb{H}.
\end{gather}
\end{Definition}

This condition means that, in the bigraded decomposition of $\Pi$ associated to \eqref{AC1}, the mixed term $\Pi_{1,1}$ is zero and $\Pi=\Pi_{2,0}+\Pi_{0,2}$, where $\Pi_{2,0}\in\Gamma(\wedge^{2}\mathbb{H})$ and $\Pi_{0,2}\in\Gamma(\wedge^{2}\mathbb{V})$ is a Poisson tensor. The characteristic distribution of $\Pi$ is contained into the (possibly non-integrable) distribution $\mathbb{H}\oplus\Pi_{0,2}^{\sharp}(\mathbb{H}^{0})$.
\begin{Remark}
Notice that conditions \eqref{AC1}, \eqref{AC3} hold whenever the following transversality and regularity conditions are satisf\/ied (in this case we say that the bivector f\/ield $\Pi$ is \emph{compatible} with the foliation $\mathcal{F}$):
\begin{gather}\label{C1}
\Pi^{\sharp}\big(\mathbb{V}^{0}\big)\cap\mathbb{V}=\{0\} ,
\end{gather}
and
\begin{gather}\label{C2}
\operatorname{rank}\Pi^{\sharp}\big(\mathbb{V}^{0}\big)=\mbox{constant on }M.
\end{gather}
Then, it follows from \eqref{C1}, \eqref{C2} that a normal bundle of $\mathcal{F}$ in \eqref{AC3} can be constructed as follows: $\mathbb{H}=\mathbb{H}'\oplus\Pi^{\sharp}(\mathbb{V}^{0})$, where $\mathbb{H}'\subset TM$ is an arbitrary sub-bundle, complementary to the regular distribution $\Pi^{\sharp}\big(\mathbb{V}^{0}\big)\oplus\mathbb{V}$.
\end{Remark}
The $\mathcal{F}$-coupling situation occurs when, along with \eqref{C1}, we have
\begin{gather*}
\operatorname{rank}\Pi^{\sharp}\big(\mathbb{V}^{0}\big)=\operatorname{codim}\mathcal{F}\mbox{ on }M,
\end{gather*}
and hence the normal bundle $\mathbb{H}$ of $\mathcal{F}$ in \eqref{AC3}, associated to $\Pi$, is unique and given by
\begin{gather}\label{CC2}
\mathbb{H}=\Pi^{\sharp}\big(\mathbb{V}^{0}\big).
\end{gather}
In this case, $\Pi$ is said to be an $\mathcal{F}$-\emph{coupling Poisson structure}.
The factorization of the Jacobi identity for $\Pi$ (see~\cite[equations (6.4)--(6.6)]{JVYV-14}) implies that the intrinsic connection $\gamma$, associated with the normal bundle \eqref{CC2}, possesses the following properties \cite{Va-04, Vo-01}: The connection~$\gamma$ is \emph{Poisson} on the Poisson bundle $(M,\mathcal{F},\Pi_{0,2})$, that is, $\mathcal{L}_{X}\Pi_{0,2}=0$ for any projectable section $X\in\Gamma_{\operatorname*{pr}}(\mathbb{H})$ (recall that the projectability property for $X$ on the foliated manifold $(M,\mathcal{F)}$ is expressed as $[X,\Gamma(\mathbb{V})]\subset\Gamma(\mathbb{V})$). Moreover, the curvature of $\gamma$ takes values in the space of Hamiltonian vector f\/ields of $\Pi_{0,2}$:
\begin{gather}\label{Cur}
R^{\gamma}(X,Y)=-\Pi_{0,2}^{\sharp}\mathrm{d}\sigma(X,Y)\quad \mbox{for all } X,Y\in\Gamma_{\operatorname*{pr}}(\mathbb{H}).
\end{gather}
Here, the $2$-form $\sigma\in\Gamma(\wedge^{2}\mathbb{V}^{0})$, called the \emph{coupling form}, is uniquely determined by $\Pi$, and satisf\/ies the $\gamma$-covariant constancy condition $\mathrm{d}_{1,0}\sigma=0$. Indeed, we can write the following explicit expression in terms of the horizontal part of $\Pi$ and the horizontal lifting induced by~$\gamma$~\cite{Vo-01}:
\begin{gather*}
\sigma (u_1,u_2)=-\big\langle \big(\Pi^\sharp_{2,0}\big|_{\mathbb{V}^0}\big)^{-1} \mathrm{hor}(u_1),\mathrm{hor}(u_2)\big\rangle,
\end{gather*}
for any $u_1,u_2\in\mathcal{X}(M)$ vector f\/ields on $M$.

The notion of $\mathcal{F}$-almost-coupling structures can be naturally generalized to the Dirac set\-ting~\cite{Va-06}. Given a Dirac structure $D\subset TM\oplus T^{\ast}M$ on the foliated manifold $(M,\mathcal{F})$, we def\/ine the tangent distribution
\begin{gather*}
H(D,\mathcal{F}):=\big\{X\in TM\,|\,\exists\,\alpha\in\mathbb{V}^{0}\mbox{ such that }(X,\alpha)\in D\big\}.
\end{gather*}
On the other hand, f\/ixing a normal bundle $\mathbb{H}$ of $\mathcal{F}$, we
consider the distributions
\begin{gather*}
D_{H}:=D\cap\big(\mathbb{H}\oplus\mathbb{V}^{0}\big)\qquad\mbox{and}\qquad D_{V}:=D\cap\big(\mathbb{V}\oplus\mathbb{H}^{0}\big).
\end{gather*}
In the general case, these are singular distributions, but the almost-coupling hypothesis implies that their respective ranks are constant, $\operatorname{rank}D_{H}=\operatorname{rank} \mathbb{H}$ and $\operatorname{rank}D_{V}=\operatorname{rank} \mathbb{V}$.

\begin{Definition}
$D$ is said to be an $\mathcal{F}$-\emph{almost-coupling Dirac structure}, via a normal bundle~$\mathbb{H}$ of~$\mathcal{F}$, if \begin{gather}\label{CP1}
H(D,\mathcal{F})\subseteq\mathbb{H}.
\end{gather}
\end{Definition}

One can show that condition \eqref{CP1} is equivalent to the following one:
\begin{gather*}
D=D_{H}\oplus D_{V}.
\end{gather*}
Also, one says that $D$ is an $\mathcal{F}$-\emph{coupling Dirac structure} \cite{Va-06,Wa} if $H(D,\mathcal{F})$ is a normal bundle of~$\mathcal{F}$,
\begin{gather*}
TM=H(D,\mathcal{F})\oplus\mathbb{V}.
\end{gather*}
Notice that \eqref{CP1} implies the following property:
\begin{gather*}
D_{V}\cap(\mathbb{V}\oplus\{0\})=\{0\}.
\end{gather*}
It follows from here that a given an $\mathcal{F}$-almost coupling Dirac structure~$D$ via~$\mathbb{H}$ induces a leaf-tangent Poisson bivector f\/ield $P\in\Gamma(\wedge^{2}\mathbb{V})$ such that
\begin{gather}
D_{V}=\big\{(P^{\sharp}\eta,\eta)\,|\,\eta\in\mathbb{H}^{0}\big\}.\label{NEW}
\end{gather}
\begin{Remark}
It easy to see that these def\/initions agree with the corresponding notions in the Poisson case. Indeed, if $D=\operatorname*{Graph}\Pi$, for a certain Poisson tensor~$\Pi$, then $H(D,\mathcal{F})=\Pi^{\sharp}(\mathbb{V}^{0})$.
\end{Remark}

The relation between geometric data and $\mathcal{F}$-coupling Dirac structures \cite{Va-06,JVYV-14,Wa}, carries over to the almost-coupling setting, with some particularities: given a connection~$\gamma$, the almost-coupling Dirac structure directly induces a leaf-tangent Poisson structure $P$, and no coupling form $\sigma$ (in the terminology of~\cite{JVYV-14}) appears. We can use that correspondence to characterize (in terms of $P$) a particular class of gauge transformations which preserve\footnote{In general, arbitrary gauge transformations do not carry almost-coupling Dirac
structures into almost-coupling structures again. This does happen for the gauge transformations~\eqref{EG1}, however.} the class of $\mathcal{F}$-almost coupling Dirac structures on a given foliated manifold $(M,\mathcal{F)}$ (this class naturally appears in the context of the
averaging method). To this end, let us pick an arbitrary $1$-form $Q\in\Gamma(\mathbb{V}^{0})$ and consider the exact gauge transformation
\begin{gather}\label{EG1}
D\mapsto\overline{D}:=\tau_{B}(D),\qquad B=-\mathrm{d}Q.
\end{gather}
Then, relative to decomposition \eqref{Spl2}, we have $B=B_{2,0}+B_{1,1}$, where $B_{2,0}=\mathrm{d}_{1,0}Q$ and $B_{1,1}=\mathrm{d}_{0,1}Q$.

\begin{Lemma}\label{lema31}
Let $D$ be an $\mathcal{F}$-almost coupling Dirac structure via a normal bundle~$\mathbb{H}$ of~$\mathcal{F}$. Then, the exact gauge transformation~\eqref{EG1} maps $D$ into the Dirac structure~$\overline{D}$ which is again $\mathcal{F}$-almost coupling, this time via the normal bundle
\begin{gather*}
\mathbb{\overline{H}}:=\big(\operatorname{Id}+P^{\sharp}\circ B_{1,1}^{\flat}\big)(\mathbb{H}).
\end{gather*}
Here, $P\in\Gamma(\bigwedge^{2}\mathbb{V})$ is the Poisson bivector field in~\eqref{NEW}.
\end{Lemma}

\begin{proof}
For an arbitrary closed $2$-form $B=B_{2,0}+B_{1,1}+B_{0,2}$, the gauge transformation $\tau_{B}(D)$ is a Dirac structure consisting of elements
of the form
\begin{gather*}
(X_{1,0}+X_{0,1})\oplus\big( \big(\alpha_{1,0}-B_{2,0}^{\flat}X_{1,0} -B_{1,1}^{\flat}X_{0,1}\big) +\big(\alpha_{0,1}-B_{1,1}^{\flat}X_{1,0}-B_{0,2}^{\flat}X_{0,1}\big)\big) ,
\end{gather*}
where $(X_{1,0}+X_{0,1},\alpha_{1,0}+$ $\alpha_{0,1})\in D$. If $B_{0,2}=0$, then taking into account~\eqref{NEW}, we can describe the characteristic distribution of $\tau_{B}(D)$ as
\begin{gather*}
H(\tau_{B}(D),\mathcal{F})=\big\{\big(\operatorname{Id}+P^{\sharp}\circ B_{1,1}^{\flat}\big)X\,|\, X\in H(D,\mathcal{F})\big\}.
\end{gather*}
This proves the statement.
\end{proof}

Notice also that the exact gauge transformation \eqref{EG1} leaves invariant the set of all $\mathcal{F}$-coupling Dirac structures on~$(M,\mathcal{F)}$~\cite{JVYV-14}.

\section{The averaging theorems for deformations}

Let $(M,\mathcal{F},P)$ be a Poisson foliation, consisting of a \emph{regular} foliation $\mathcal{F}$ on $M$, and a leaf-tangent Poisson bivector f\/ield $P\in$ $\Gamma(\wedge^{2}\mathbb{V})$. Recall that we denote by $\mathbb{V}=T\mathcal{F}\subset TM$ the tangent bundle of $\mathcal{F}$ and by $\mathbb{V}^{0}\subset T^{\ast}M$ its annihilator. The Poisson tensor $P$ is characterized by the property that the symplectic leaf of $P$ through each point $q\in M$, is contained into the leaf $\mathcal{F}_{q}$ of the regular foliation. Another characterization, given in \cite{Va-06}, states that the leaves of~$\mathcal{F}$ are Poisson submanifolds.

Consider a smooth (left) action on $M$, $\Phi\colon G\times M\rightarrow M$, of a connected and compact Lie group~$G$, which is \emph{compatible} with the Poisson structure $P\in$ $\Gamma(\wedge ^{2}\mathbb{V})$, in the sense of Def\/inition~\ref{def2.1}: For every~$a\in\mathfrak{g}$, the inf\/initesimal generator $a_{M}$ has the form
\begin{gather}\label{AT}
a_{M}=P^{\sharp} \mu_{a}\quad \mbox{for all }a\in\mathfrak{g},
\end{gather}
for a certain $\mathbb{R}$-linear mapping $\mu\in\operatorname*{Hom}(\mathfrak{g},\Omega^{1}(M))$. In particular, this property means that the
$G$-action preserves the leaves of the symplectic foliation of $P$, and condition~\eqref{LW1} holds, that is, the $G$-action is compatible with the
associated Dirac structure $D=\operatorname*{Graph}P$.

A $G$-action on the Poisson foliation is said to be \emph{locally Hamiltonian}, if each inf\/initesimal generator $a_{M}$ is a locally Hamiltonian
vector f\/ield relative to $P$, in other words, one can choose the $1$-form $\mu_{a}$ in~\eqref{AT} to be closed on~$M$,
\begin{gather}\label{LH}%
\mu_{a}\ \in\Omega_{\operatorname*{cl}}^{1}(M)\quad \mbox{for all }a\in\mathfrak{g}.
\end{gather}

Suppose now that we start with a \emph{smooth deformation} $\{\Pi_\varepsilon\}_{\varepsilon\in [0,1]}$ of the \emph{leaf-tangent Poisson structure} $P$, so each $\Pi_{\varepsilon}$ is a Poisson tensor on $M$, the whole family is smoothly dependent in $\varepsilon$, and is such that $\Pi_{0}=P$. Moreover, we assume that, for each $\varepsilon\in [0,1]$, the bivector f\/ield $\Pi_{\varepsilon}$ is an $\mathcal{F}$-almost-coupling Poisson structure via an $\varepsilon$-independent normal bundle~$\mathbb{H}$ of~$\mathcal{F}$:
\begin{gather}\label{O1}
\Pi_{\varepsilon}^{\sharp}\big(\mathbb{V}^{0}\big)\subseteq\mathbb{H}\quad \mbox{for all }\varepsilon\in [0,1].
\end{gather}
Such a family $\{\Pi_{\varepsilon}\}$, will be called an $\mathcal{F}$-\emph{almost-coupling Poisson deformation} of $P$ via $\mathbb{H}$. It
follows from~\eqref{O1} that each $\Pi_{\varepsilon}$ admits the following bigraded decomposition relative to~\eqref{AC1}: $\Pi_{\varepsilon}=(\Pi_{\varepsilon})_{2,0}+(\Pi_{\varepsilon})_{0,2}$, where $(\Pi_0)_{0,2}=P$. We will also assume that the leaf-tangent component of $\Pi_{\varepsilon}$ of bidegree $(0,2)$ is independent of $\varepsilon$ and hence
\begin{gather*}
(\Pi_{\varepsilon})_{0,2}=P\quad \mbox{for all }\varepsilon\in [0,1].
\end{gather*}
Then, $(\Pi_{\varepsilon})_{2,0}=\varepsilon\Lambda_{\varepsilon}$ for a~certain bivector f\/ield $\Lambda_{\varepsilon}$ of bidegree $(2,0)$, smoothly varying in $\varepsilon$, and therefore, the deformation of $P$ can be parameterized as
\begin{gather}\label{SD2}
\Pi_{\varepsilon}=P+\varepsilon\Lambda_{\varepsilon},\qquad \Lambda_{\varepsilon}\in\Gamma\big({\wedge}^{2}\mathbb{H}\big).
\end{gather}

For a f\/ixed $\mathbb{R}$-linear mapping
$\mu\in\operatorname*{Hom}(\mathfrak{g},\Omega^{1}(M))$ in \eqref{AT}, we can decompose $\mu=\mu_{1,0}+\mu_{0,1}$ relative to splitting \eqref{Spl2},
associated to a f\/ixed normal bundle $\mathbb{H}$ in \eqref{O1}.

Finally, let us associate to the family $\{\Pi_{\varepsilon}\}$ (as in \eqref{SD2}) a~smooth $\varepsilon$-dependent family of Dirac structures
\begin{gather}\label{D1}%
D_{\varepsilon}:=\operatorname*{Graph}\Pi_{\varepsilon}=\operatorname*{Graph} P\oplus\operatorname*{Graph}\varepsilon\Lambda_{\varepsilon},
\end{gather}
which is to be viewed as a deformation of the `limiting' Dirac structure $D_{0}=\operatorname*{Graph}\Pi_{0}=\operatorname*{Graph}P$. Then, for each
$\varepsilon$, the Dirac structure $D_{\varepsilon}$ also satisf\/ies the $\mathcal{F}$-almost-coupling condition~\eqref{CP1}. Under the above hypotheses, we are in position to state our main results regarding the existence of the $G$-average for these deformations of Dirac and Poisson structures. The f\/irst one is a consequence of Lemma~\ref{lema22}.
\begin{Theorem}
For every $\varepsilon\in\lbrack0,1]$, the $G$-average $\overline{D}_{\varepsilon}$, of the $\mathcal{F}$-almost-coupling Dirac structure~$D_{\varepsilon}$ in~\eqref{D1}, is well-defined and given by the exact gauge transformation
\begin{gather}\label{AD1}
\overline{D}_{\varepsilon}=\{(X,\alpha+\mathbf{i}_{X}\mathrm{d}\Theta)\,|\,(X,\alpha)\in D_{\varepsilon}\},
\end{gather}
where $\Theta$ is the $1$-form defined by \eqref{MF}, with $\rho=$ $\mu_{0,1}$.
\end{Theorem}
\begin{proof}
It follows from \eqref{SD2} that, for every $\varepsilon\in [0,1]$ and $a\in\mathfrak{g}$, we have
\begin{gather*}
\Pi_{\varepsilon}^{\natural}(\mu_{a})_{0,1}=(\Pi_{\varepsilon})_{0,2}^{\natural}(\mu_{a})_{0,1}=(\Pi_{\varepsilon})_{0,2}^{\natural}\mu
_{a}=P^{\natural}\mu_{a}=a_{M},
\end{gather*}
and, hence, the $G$-action preserves the symplectic leaves of $\Pi_{\varepsilon}^{\natural}$, and the compatibility condition~\eqref{LW1} holds for
$\rho=\mu_{0,1}$. Then, by Lemma \ref{lema22}, the $G$-average of $D_{\varepsilon}$ is well-def\/ined and given by~\eqref{AD1}.
\end{proof}

\begin{Theorem}\label{teo42}
If the $G$-action is locally Hamiltonian $($condition~\eqref{LH}$)$, then, for any $G$-invariant open connected set with compact closure, $N\subset M$, and for sufficiently small $\varepsilon\in {]0,1[}$, the restriction of the Dirac structure~\eqref{AD1}, $\overline{D}_{\varepsilon}|_N$, is the graph of a~$G$-invariant Poisson structure~$\overline{\Pi}_{\varepsilon}$ on~$N$ defined by
\begin{gather}\label{AP1}
\overline{\Pi}_{\varepsilon}^{\sharp}=\Pi_{\varepsilon}^{\sharp}\circ \big(\operatorname{Id}+(\mathrm{d}Q)^{\flat}\circ\Pi_{\varepsilon}^{\sharp}\big)^{-1}
\end{gather}
with $\overline{\Pi}_{0}=P$. Here $Q\in\Gamma(\mathbb{V}^{0})$ is expressed as
\begin{gather}\label{AP2}
Q:=-\int_G \left( \int^1_0 \Phi_{\exp(\tau a)}^{\ast}(\mu_{a})_{1,0}\,\mathrm{d}\tau\right)\mathrm{d}g,\qquad g=\exp a.
\end{gather}
Moreover, the Poisson structure $\overline{\Pi}_{\varepsilon}$ is $\mathcal{F}$-almost-coupling via the following $G$-invariant normal bundle of~$\mathcal{F}$:
\begin{gather}\label{IN}
\mathbb{\overline{H}}:=\operatorname*{Span}\big\{\bar{X}=X+P^{\sharp}\mathrm{d}Q(X)\,|\, X\in\Gamma_{\operatorname*{pr}}(\mathbb{H})\big\}.
\end{gather}
\end{Theorem}
\begin{proof}
Under the hypothesis that the $G$-action is locally Hamiltonian, let us f\/ix a $G$-invariant relatively compact domain $N\subseteq M$. Let $B=-\mathrm{d}\Theta$, with~$\Theta$ constructed as in~\eqref{MF}. By Lemma~\ref{lem21} it suf\/f\/ices to show that there exists a $0<\delta <1$ such that
\begin{gather}
\big( \operatorname{Id}-B^{\flat}\circ\Pi_{\varepsilon}^{\natural}\big) \mbox{ is invertible on }N\mbox{ for }\varepsilon\in
[0,\delta].\label{AP3}
\end{gather}
In terms of the bigraded components of $\mu$, the closedness condition $\mathrm{d}\mu=0$ splits into
\begin{gather*}
\mathrm{d}_{0,1}\mu_{0,1}=0,\qquad \mathrm{d}_{1,0}\mu_{0,1}=-\mathrm{d}_{0,1}\mu_{1,0},
\qquad \mathrm{d}_{2,-1}\mu_{0,1}=-\mathrm{d}_{1,0}\mu_{1,0}.
\end{gather*}
By using these relations, we see from \eqref{MF} that $B=-\mathrm{d}\Theta=-\mathrm{d}Q$, where $Q\in\Gamma(\mathbb{V}^{0})$ is given by~\eqref{AP2}. It follows that $B=B_{2,0}+B_{1,1}$, where $B_{2,0}=-\mathrm{d}_{1,0}Q$ and $B_{1,1}=-\mathrm{d}_{0,1}Q$. From here, taking into account~\eqref{SD2}, for an arbitrary $\alpha=\alpha_{1,0}+\alpha_{0,1}\in\Omega^{1}(M)$, we get
\begin{gather*}
\big(B^{\flat}\circ\Pi_{\varepsilon}^{\natural}\big)(\alpha_{1,0}+\alpha _{0,1})=B_{1,1}^{\flat}\circ P^{\natural}\alpha_{0,1}+\varepsilon\big(
B_{2,0}^{\flat}\circ(\Lambda_{\varepsilon})_{2,0}^{\natural}\alpha _{1,0}+B_{1,1}^{\flat}\circ(\Lambda_{\varepsilon})_{2,0}^{\natural}\alpha_{1,0}\big) .
\end{gather*}
This shows that the matrix of the morphism $\operatorname{Id}-B^{\flat}\circ\Pi_{\varepsilon}^{\natural}\colon T^{\ast}M\rightarrow T^{\ast}M$, in a local
basis compatible with the splitting~\eqref{Spl2}, has the form
\begin{gather*}
\begin{pmatrix}
I & \ast\\
0 & I
\end{pmatrix}
 +O(\varepsilon).
\end{gather*}
This fact, together with a standard compactness argument, proves~\eqref{AP3}. Furthermore, it follows that
\begin{gather}\label{R1}%
\big(\operatorname{Id}+(\mathrm{d}Q)^{\flat}\circ P^{\sharp}\big)^{-1}=\operatorname{Id}-(\mathrm{d}Q)^{\flat}\circ P^{\sharp},
\end{gather}
and
\begin{gather}\label{R2}
P^{\sharp}\circ(dQ)^{\flat}\circ P^{\sharp}=0.
\end{gather}
From here and \eqref{AP1}, we conclude that $\overline{\Pi}_{0}=P$. Finally, to show that the Poisson structure~$\overline{\Pi}_{\varepsilon}$ is~$\mathcal{F}$-almost coupling, we consider the connection $\gamma$ associated to the normal bundle $\mathbb{H}$ and its $G$-average $ \langle\gamma \rangle^{G}$. Then, the normal bundle $\ker \langle\gamma \rangle^{G}$ just coincides with $\overline{\mathbb{H}}$ in~\eqref{IN} (see~\cite{JVYV-14}). Applying Lemma~\ref{lema31} ends the proof.
\end{proof}

Consider now the case of $\Pi_{\varepsilon}$ a coupling Poisson structure for $\varepsilon\neq0$, that is, $\mathbb{H}= \Pi_{\varepsilon}^{\sharp}(\mathbb{V}^{0})$ is a~normal bundle of~$\mathcal{F}$. Then, the corresponding Dirac structure is represented as
\begin{gather*}
D_{\varepsilon}=\operatorname*{Graph}(\Pi_{\varepsilon})=\big\{\big(\varepsilon
X+P^{\sharp}\alpha,\alpha-\mathbf{i}_{X}\sigma_{\varepsilon}\big)\,|\, X\in \Gamma(\mathbb{H}),\ \alpha\in\Gamma\big(\mathbb{H}^{0}\big)\big\},
\end{gather*}
where the $2$-form $\sigma_{\varepsilon}\in\Gamma(\wedge^{2}\mathbb{V}^{0})$ smoothly depends in $\varepsilon$, satisf\/ies the $\mathbb{H}$-covariant constancy condition $\mathrm{d}_{1,0}\sigma_{\varepsilon}=0$, and has the following expansion around $\varepsilon=0$:
\begin{gather*}
\sigma_{\varepsilon}=c+\varepsilon\sigma+O\big(\varepsilon^{2}\big).
\end{gather*}
Here, the $2$-form $c\in\Gamma(\wedge^{2}\mathbb{V}^{0})$ takes values in the Casimir functions of $P$,
\begin{gather*}
c(X,Y)\in\operatorname*{Casim}(M,P)\quad \mbox{for all }X,Y\in\Gamma_{\operatorname*{pr}}(\mathbb{H}),
\end{gather*}
and the $2$-form $\sigma\in\Gamma(\wedge^{2}\mathbb{V}^{0})$ satisf\/ies the curvature identity \eqref{Cur} with $\Pi_{0,2}=P$ (for more details, see \cite{Va-04,JVYV-14}). From these remarks and Theorem \ref{teo42}, we deduce the following consequence.

\begin{Corollary}
In the coupling case, for a locally Hamiltonian $G$-action, the $G$-average $\overline{D}_{\varepsilon}$ is an $\mathcal{F}$-coupling Dirac structure via the invariant normal bundle $\mathbb{\bar{H}}$ given by
\begin{gather*}
\overline{D}_{\varepsilon}=\big\{\big(\varepsilon X+P^{\sharp}\alpha,\alpha-\mathbf{i}
_{X}\sigma_{\varepsilon}+\mathbf{i}_{\varepsilon X+P^{\sharp}\alpha}\mathrm{d}Q\big)\,|\,
X\in\Gamma(\mathbb{H}),\ \alpha\in\Gamma\big(\mathbb{H}^{0}\big)\big\}.
\end{gather*}
In particular, $\overline{D}_{0}=\operatorname*{Graph}(P)$.
\end{Corollary}

A natural question concerns the relationship between the original, $\varepsilon$-dependent, Poisson structure $\Pi_{\varepsilon}$, and its $G$-average $\overline{\Pi}_{\varepsilon}$, as $\varepsilon$ tends to~$0$.

\begin{Theorem}\label{teo43}
Let $\overline{\Pi}_{\varepsilon}$ be the $G$-average of the $\mathcal{F}$-almost-coupling Poison structure $\Pi_{\varepsilon}$ defined in~\eqref{AP1}. Let $N\subset M$ be a $G$-invariant open subset whose closure is compact. There exists a~smooth isotopy $\phi_{\varepsilon}\colon N\rightarrow M$, such that, for $\varepsilon$ sufficiently small,
\begin{gather*}
\phi_{\varepsilon}^{\ast}\Pi_{\varepsilon}=\overline{\Pi}_{\varepsilon},\qquad \phi_{0}=\operatorname*{Id}
\end{gather*}
on $N$.
\end{Theorem}

\begin{proof} Following the reasoning in the proof of Theorem~\ref{teo42}, we can show that there exists a~$\delta>0$ such that the gauge transformation of~$\Pi_{\varepsilon}$, associated to the exact $2$-form $tB=-t\mathrm{d}Q$, exists for all $\varepsilon \in [0,\delta]$, $t\in [0,1]$, and gives the $2$-parameter family of Poisson structures on~$N$,~$\Pi_{\varepsilon,t}$, characterized by
\begin{gather}\label{PE}
\Pi^\sharp_{\varepsilon,t}=\Pi_{\varepsilon}^{\sharp}\circ\big(\operatorname{Id}+t(\mathrm{d}Q)^{\flat}\circ\Pi_{\varepsilon}^{\sharp}\big)^{-1}.
\end{gather}
Then, $\Pi_{\varepsilon,0}=\Pi_{\varepsilon}$ and $\Pi_{\varepsilon,1} =\overline{\Pi}_{\varepsilon}$. Fixing $\varepsilon\in\lbrack0,\delta]$, one can verify~\cite{FrMa13,JVYV-14} that the time-dependent vector f\/ield on $N$ given by
\begin{gather}\label{He1}
Z_{\varepsilon,t}=-\Pi_{\varepsilon,t}^{\sharp}(Q)=-\Pi_{\varepsilon}^{\sharp} \circ\big(\operatorname{Id}+t(\mathrm{d}Q)^{\flat}\circ\Pi_{\varepsilon}^{\sharp}\big)^{-1}(Q)
\end{gather}
satisf\/ies the homotopy equation (where $[\![\cdot ,\cdot ]\!]$ denotes the Schouten bracket for multivector f\/ields on~$M$~\cite{Va-94})
\begin{gather}\label{He2}%
[\![ Z_{\varepsilon,t},\Pi_{\varepsilon,t}]\!]= -\frac{\mathrm{d}\Pi_{\varepsilon,t}}{\mathrm{d}t}.
\end{gather}
Moreover, we have $\Pi_{0,t}=P$ and $Z_{0,t}=-P^{\sharp}(Q)=0$. These facts, together with the compactness of the closure of~$N$, show that (by shrinking $\delta >0$ if necessary), for each $\varepsilon\in [0,\delta]$, the f\/low~$\operatorname*{Fl}_{Z_{\varepsilon,t}}^{t}$ of $Z_{\varepsilon,t}$ is well-def\/ined on $N$, for all $t\in [0,1]$. Then, it suf\/f\/ices to put $\phi_{\varepsilon}= \operatorname*{Fl}_{Z_{\varepsilon,t}}^{t}|_{t=1}$.
\end{proof}

\section{Inf\/initesimal deformations}

In this section, we will derive a f\/irst-order approximation formula for the averaged Poisson structure in \eqref{AP1}.

Let $\Pi_{\varepsilon}=P+\varepsilon\Lambda_{0}+O(\varepsilon ^{2})$ be a family of almost-coupling Poisson structures on $M$. From the Jacobi identity $[\![\Pi_{\varepsilon},\Pi_{\varepsilon}]\!]=0=[\![P,P]\!]$, and it follows that the bivector f\/ield $\Lambda_{0}\in\Gamma(\wedge^{2}\mathbb{H})$ is a~$2$-cocycle in the Lichnerowicz--Poisson complex of $(M,P)$, $[\![P,\Lambda_{0}]\!]=0$. This $2$-cocycle $\Lambda_{0}$, determines the inf\/initesimal (f\/irst-order) part of the almost-coupling deformation of $P$. Then, taking into account the identities~\eqref{R1} and~\eqref{R2}, we deduce from~\eqref{AP1}
\begin{gather}\label{He3}
\overline{\Pi}_{\varepsilon}=P+\varepsilon\overline{\Lambda}_{0}+O\big(\varepsilon^{2}\big),
\end{gather}
where $\overline{\Lambda}_{0}$ is another $G$-invariant $2$-cocycle in the Lichnerowicz--Poisson complex, given by
\begin{gather*}
\overline{\Lambda}_{0}^{\natural}=\big( \operatorname{Id}-P^{\sharp}\circ(\mathrm{d}Q)^{\flat}\big) \circ\Lambda_{0}^{\natural}\circ\big(
\operatorname{Id}-(\mathrm{d}Q)^{\flat}\circ P^{\sharp}\big).
\end{gather*}
By varying \eqref{He3} with respect to $\varepsilon$, we get the following inf\/initesimal version of Theorem~\ref{teo43}.

\begin{Proposition}
The cohomology classes of the $2$-cocycles $\Lambda_{0}$ and $\overline{\Lambda}_{0}$, coincide.
\end{Proposition}

\begin{proof}
For the inf\/initesimal generator $Z_{\varepsilon,t}$ in \eqref{He1}, and the family of Poisson structures $\Pi_{\varepsilon,t}$ in \eqref{PE}, we can evaluate their expansion around $\varepsilon =0$: $Z_{\varepsilon,t}=\varepsilon W_{t}+O(\varepsilon^{2})$ and $\Pi_{\varepsilon,t}=P+\varepsilon\Psi _{t}+O(\varepsilon^{2})$, where $W_{t}$ and $\Psi_{t}$ are a time-dependent vector f\/ield and a bivector f\/ield, respectively. Putting these expansions into \eqref{He2}, and collecting all f\/irst-order terms in $\varepsilon$, leads to the equation $[\![W_{t},P]\!]=-\frac{\mathrm{d}\Psi_{t}}{\mathrm{d}t}$. Integrating this equation with respect to $t$, and taking into account that $\Psi_{0}=\Lambda_{0}$, and $\Psi_{1}=\overline{\Lambda}_{0}$, we get $\overline{\Lambda}_{0}-\Lambda_{0}=[\![w,P]\!]$, where $w=-\int_{0}^{1}W_{t}\,\mathrm{d}t$.
\end{proof}
\begin{Remark}\label{lastrem}
These results can be used to construct normal forms for perturbed dynamics associated to almost-coupling deformations of foliated Poisson
manifolds with symmetry. Typically, such perturbed dynamics appear in the context of adiabatic theory~\cite{ArKN-88} on nontrivial phase
spaces, particularly those in which no global action-angle coordinates can be introduced \cite{MJYu-13,MJYu-13a,MYU-16}.
\end{Remark}
We have a simple example of an $\mathcal{F}$-almost-coupling Poisson deformation when the foliation $\mathcal{F}$ is a f\/ibration.
\begin{Example}\label{examp51}
Suppose we start with a Poisson f\/iber bundle $(\pi\colon M\rightarrow S,P)$ whose base $S$ carries a Poisson structure $\psi\in\Gamma(\wedge^{2}TS)$. Assume that there exists a f\/lat Poisson connection~$\gamma$ on~$M$, associated to an integrable distribution $\mathbb{H}$, which is complementary to the vertical one $\mathbb{V}=\ker \mathrm{d}\pi$. Let $\operatorname{hor}^{\gamma}(\psi)\in \Gamma(\wedge^{2}\mathbb{H})$ be the $\gamma$-horizontal lift of the Poisson bivector f\/ield~$\psi$. Then, by using the standard properties of the Schouten bracket (see~\cite{Va-94}), one can show that the bivector f\/ield
\begin{gather}\label{FC}
\Pi_{\varepsilon}=P+\varepsilon\operatorname{hor}^{\gamma}(\psi)
\end{gather}
satisf\/ies the Jacobi identity and gives an almost-coupling Poisson structure for every $\varepsilon\in [0,1]$. It is clear that $\Pi_{\varepsilon}$ is a coupling Poisson tensor on $M$ if and only if the bivector f\/ield $\psi$ is non-degenerate, and hence induces a symplectic form on $S$.
\end{Example}

Finally, let us say some words about the physical meaning of the preceding constructions. Consider the Hamiltonian vector f\/ield
$X(\varepsilon)=\mathbf{i}_{\mathrm{d}F}P+\varepsilon\mathbf{i}_{\mathrm{d}F} \operatorname{hor}^{\gamma}(\psi)$, relative
to a function \smash{$F\in C^{\infty}(M)$} and the Poisson structure~\eqref{FC}. As stated in the Introduction, the corresponding dynamical system
\begin{gather}\label{SL-F1}
\dot{\xi}=\varepsilon\mathbf{i}_{\mathrm{d}F}\operatorname{hor}^{\gamma}(\psi),\qquad \dot{x}=\mathbf{i}_{\mathrm{d}F}P,
\end{gather}
belongs to the class of the so-called \emph{slow-fast Hamiltonian systems}~\cite{ArKN-88}, where the coordinates \smash{$\xi=(\xi^{i})$}, along the base~$S$, and the coordinates $x=(x^{\alpha})$, along the f\/ibers of~$\pi$, are called the slow and fast variables, respectively. We observe that the inf\/initesimal generator of the Poisson isotopy~$\phi_{\varepsilon}$ in Theorem~\ref{teo43}, controls the `geometric' part of an invariant normalization transformation for the system~\eqref{SL-F1} as~$\varepsilon$ tends to zero. As stated in Remark~\ref{lastrem}, these transformations, and the normal forms to which they lead, are crucial for determining asymptotic properties of the system, such as stability, existence of periodic orbits, or the computation of adiabatic invariants, see~\cite{MJYu-13} for some examples.

\subsection*{Acknowledgements}

We express our gratitude to the anonymous referees, whose detailed comments and criticism have greatly improved the contents of this paper. Also, we thank the organizers of the Gone Fishing Meeting at Berkeley (2015), for the opportunity of presenting and discussing some of the results exposed here. This work was partially supported by the Mexican National Council of Science and Technology (CONACyT), under research projects CB-2012-179115 (JAV) and CB-2013-219631 (YuV).

\pdfbookmark[1]{References}{ref}
\LastPageEnding

\end{document}